\documentclass[12pt]{amsart}
\usepackage{graphicx}
\usepackage{verbatim}
\usepackage{tikz}
\usetikzlibrary{arrows}
\usetikzlibrary{calc}

\usepackage{mathpazo}
\usepackage[scaled=.95]{helvet}
\usepackage{courier}

\advance \topmargin by -\headheight
\advance \topmargin by -\headsep
\evensidemargin \oddsidemargin
\theoremstyle{plain}
\newtheorem{thm}{Theorem}[section]
\newtheorem{prop}[thm]{Proposition}
\newtheorem{lemma}[thm]{Lemma}
\newtheorem{cor}[thm]{Corollary}
\newtheorem{conj}{Conjecture}

\newtheorem*{question}{Question}

\theoremstyle{definition}

\theoremstyle{remark}

\title{Aspherical $4$-manifolds of odd Euler characteristic}
\author{Allan L. Edmonds}
\address{Department of Mathematics, Indiana University, Bloomington, IN 47405}
\email{edmonds@indiana.edu}
\thanks{}
\begin{document}
\begin{abstract}
An explicit construction of closed, orientable, smooth, aspherical 4-manifolds with any odd Euler characteristic greater than 12 is presented. The manifolds constructed here are all Haken manifolds in the sense of B. Foozwell and H. Rubinstein and can be systematically reduced to balls by suitably cutting them open along essential codimension-one submanifolds. It is easy to construct examples with even Euler characteristic from products of surfaces. And Euler characteristics divisible by 3 are know to arise from complex algebraic geometry considerations. Examples with Euler characteristic 1, 5, 7, or 11 appear to be unknown. 
\end{abstract}
\maketitle
\setcounter{tocdepth}{1}
\tableofcontents
\section{Introduction}

The simplest algebraic-topological invariant of a space is its Euler characteristic. Now any integer can be the Euler characteristic of a closed, orientable $4$-manifold, as one can easily see by forming connected sums of standard familiar manifolds $S^{4}$, $\mathbb{C}P^{2}$, and $T^{4}$. But the Euler characteristic of an aspherical $4$-manifold seems to be highly constrained.\footnote{A path-connected space $X$ is said to be \emph{aspherical} if its higher homotopy groups $\pi_{n}(X)$ vanish for $n\ge 2$. Basic examples include products of surfaces of positive genus. }

It is a well-known conjecture  that the Euler characteristic of  a closed, aspherical $4$-manifold is non-negative. 
One can easily realize any Euler characteristic divisible by 4 by a product of two closed orientable surfaces of positive genus, and one can realize any non-negative integer at all by a product of non-orientable surfaces.  One can then realize any non-negative even integer by the orientable double cover of a product of two non-orientable surfaces. But it seems more difficult to find examples of orientable manifolds with odd Euler characteristic.

\begin{thm}[Main Theorem]
For any odd integer $n\ge 13$, there is a closed, oriented, aspherical, smooth $4$-manifold $X_{n}$ with Euler charactertistic $\chi(X_{n})=n$.
\end{thm}
The manifolds constructed here are all Haken manifolds in the sense of Foozwell and Rubinstein \cite{FoozwellRubinstein2011} in that they can be systematically reduced to balls by suitably cutting them open along essential codimension-one submanifolds. 

The only other examples of explicit closed, orientable, aspherical $4$-manifolds of  odd Euler characteristic known to the author arise from algebraic geometry.\footnote{See Subsection \ref{subsec:hyperbolization} for some less explicit examples with  odd $\chi$, whose actual Euler characteristics are unclear.} They have Euler characteristics $\chi=3n$ for $n\geq 1$.  See Section \ref{subsec:deeper}. It thus remains a challenge to construct  4-manifolds with other odd Euler characteristics less than 12 (in particular 1, 5, 7, 11), and we offer the following conjecture.

\begin{conj}\label{conj:chi=1}
There is a closed oriented aspherical $4$-manifold $X^{4}$ with Euler characteristic $\chi(X^{4})=1$.
\end{conj}

The main construction uses certain basic building blocks of aspherical manifolds with standard aspherical and $\pi_{1}$-injective boundary.
The starting point is the product $M^{2}\times M^{2}$ where $M^{2}$ is a once-punctured torus. The sides of this manifold are then partially identified to create a manifold with $\chi=1$ and boundary a certain torus bundle over the circle. We then argue that the torus bundle in question also bounds an aspherical $4$-manifold built out of a piece with $\chi=0$ and three with  $\chi=4$. This step depends on analyzing and tweaking a  result of Foozwell and Rubinstein \cite{FoozwellRubinstein2016}, and yields the example with $\chi=13$. The more general result is an elaboration involving more pieces. 

The main steps in the proof are stated and the overall argument is given in Section \ref{sec:blocks}. The steps involve pieces we call Cores, Splitters, Connectors, and Caps, which are glued together along boundary components to create the required manifolds. The individual steps are completed in subsequent sections. 


As part of the constructions discussed in this paper one tries to find for a closed oriented Haken $3$-manifold $M^{3}$ a compact, oriented Haken $4$-manifold $W^{4}$ with $\partial W^{4}=M^{3}$,  (with boundary understood to be $\pi_{1}$-injective)  and Euler characteristic as small as possible.

\begin{conj}\label{conj:E=0}
For any closed oriented Haken $3$-manifold $M^{3}$ there is a compact oriented Haken $4$-manifold $W^{4}$ with $\partial W^{4}=M^{3}$ (and $\pi_{1}$-injective boundary) and Euler characteristic $\chi(W^{4})=0$.
\end{conj}
The construction of this paper shows that Conjecture \ref{conj:E=0} implies Conjecture \ref{conj:chi=1}.

Less boldly one might simply define the \emph{minimum Euler invariant} of any closed, oriented aspherical $3$-manifold $M^{3}$ by
\[
\mathcal{E}(M^{3}) = \min\{\chi(W^{4}): \partial W^{4}=M^{3}\}
\]
where it is understood that $W^{4}$ ranges over compact, connected, oriented aspherical manifolds  with $\pi_{1}$-injective boundary $M^{3}$.

\begin{question}
For any closed, oriented aspherical $3$-manifold $M^{3}$, what is $\mathcal{E}(M^{3})$? Is  there a  $3$-manifold $M^{3}$ such that $\mathcal{E}(M^{3})\ne 0$? 
\end{question}

The results in this paper would show that for a torus bundle over the circle we have $\mathcal{E}\le 20$.

\subsubsection*{Acknowledgements} Thanks to Mike Davis, Chuck Livingston, Ron Stern, and Matthew Stover for useful comments and pointers to the literature at various stages in this investigation.

\section{Brief survey of related results}

\subsection{Simplest examples}
Examples of closed, orientable, aspherical $4$-manifolds arising from aspherical surfaces.

$\chi=0$: $4$-torus $T^{4}$, and many others, e.g. $M^{3}\times S^{1}$, where $M^{3}$ is any aspherical $3$-manifold.

$\chi=4n$: product of two orientable, aspherical surfaces. $\chi(T_{g}\times T_{h})=4(1-g)(1-h)$

$\chi=2n$: orientable double cover $(U_{g}\times U_{h})^{\sim}$ of a product of two non-orientable, aspherical surfaces, where $\chi(U_{g}\times U_{h})=(2-g)(2-h)$. (Of course, among non-orientable manifolds $U_{g}\times U_{h}$ realizes any non-negative Euler characteristic.)

\subsection{Deeper examples}
\label{subsec:deeper}
$\chi=2$: Aspherical rational homology $4$-sphere (Luo \cite{Luo1988}). Later integral homology $4$-spheres were obtained by Ratcliffe and Tschantz \cite{RatcliffeTschantz2005}. Of course $\chi=2$ is also realized as the orientable double cover of the product of two non-orientable surfaces of Euler characteristic $-1$.

$\chi=3n$: Mumford's  \cite{Mumford1979} fake projective plane, with the rational homology of $\mathbb{C}P^{2}$. This is part of a family of $50$ such manifolds (up to homeomorphism) arising in algebraic geometry as certain complex surfaces of general type that appear as ball quotients. It is known that all of these manifolds have finite nontrivial integral first homology groups.  It is also known that these fake projective planes are not Haken. This follows from a result of Stover \cite{Stover2007,Stover2013} who showed that their fundamental groups cannot split as a free product with amalgamation (and also do not admit a surjection onto $\mathbb{Z}$).

In addition there is the Cartwright-Steger manifold $S_{1}$ \cite{CartwrightSteger2010} (discovered in the process of classifying fake projective planes) with $\chi=3$, which is also a ball quotient, but additionally has the property that its first homology is infinite. Passing to covering spaces we therefore obtain aspherical $4$-manifolds $S_{n}$ with $\chi=3n$ for all positive integers $n$.  The Cartwright-Steger manifold (as well as its finite covers) appears likely to be Haken, since it can be expressed as a fibering with singularities over a torus with generic fiber a surface of genus 19.

Other examples arise from Hirzebruch's construction \cite{Hirzebruch1983} of abelian branched covers of the complex projective plane, branched along suitable arrangements of lines. These also have Euler characteristics divisible by 3.

Yet earlier examples of complex surfaces of general type and  $c_{1}^{2}=3c_{2}$ (implying $\chi=3\sigma$) are said to have been known to Borel as early as 1963.

\subsection{Hyperbolization}\label{subsec:hyperbolization}
The process of Gromov hyperbolization is well known to produce aspherical manifolds with control on signature, but little control on the Euler characteristic.

\begin{prop}
For every integer $s$ there exists a closed, oriented, aspherical $4$-manifold $X$ with signature $\sigma(X)=s$.
\end{prop}
\begin{proof}[Proof sketch]
Following Davis and Januszkiewicz  \cite{DavisJanuszkiewicz1991} one simply applies a suitable hyperbolization process to $\#_{s} \mathbb{C}P^{2}$, which preserves oriented bordism class, hence signature.
\end{proof}

If $s$ is odd this produces a closed aspherical $4$-manifold $X$ with $\chi(X)$ odd, since $\chi\equiv \sigma\mod 2$. In this approach it is far from clear exactly which  integers can be obtained as Euler characteristics. In practice the Euler characteristic seems quite large.

\section{Basic formulas and facts}
\subsection{Aspherical unions}
We use repeatedly the following standard result.
\begin{thm}[J.H.C. Whitehead \cite{Whitehead1939}]
Suppose $Z$ is a finite cell complex expressed as a union $X\cup Y$ of subcomplexes such that $X$, $Y$, and each component of $X\cap Y$ are aspherical and each component of $X\cap Y$ is $\pi_{1}$-injective in $X$ and $Y$. Then $Z$ is also aspherical. 
\end{thm}
A direct prooof constructs the universal covering of $Z$ out of copies of the universal coverings of $X$ and $Y$ glued together along copies of the universal covering of $X\cap Y$.

We will apply this result in a situation where $Z$ is obtained from an aspherical  space $X$ by identifying two aspherical and $\pi_{1}$-injective subspaces $Y_{1}$ and $Y_{2}$. One reduces this situation to the preceding one by viewing it up to homotopy as being obtained by attaching $Y_{1}\times I$ to $X$.

\subsection{Euler characteristic}
We will need a few simple standard facts. Here we understand the spaces to have the homotopy types of finite complexes.
\subsubsection*{Sum formula}If $Z=X\cup Y$, then  $\chi(Z)=\chi(X)+\chi(Y)-\chi(X\cap Y)$.

\subsubsection*{Product formula}If $Z=X\times Y$, then $\chi(Z)=\chi(X)\chi(Y)$.

Of course the Euler characteristic of an oriented manifold is independent of any choice of orientation, so that  in that case $\chi(-X)=\chi(X)$.

\section{Building blocks and overall construction}
\label{sec:blocks}
Here we state four propositions whose proofs are deferred to subsequent sections, and use the propositions to give a proof of the main theorem. Our starting point is the following construction.

\begin{prop}[Core]\label{prop:core}
There is a compact, oriented, smooth, aspherical $4$-manifold $\rm{Core}$ with $\chi=1$. The aspherical and $\pi_{1}$-injective boundary  of $\rm{Core}$ is the torus bundle over the circle $T^{2}(\Phi)$ with monodromy map  given by 
$$\Phi =  \left(\begin{array}{rr} 0 & 1 \\ -1& 0 \end{array}\right).$$
\end{prop}

We adopt the following notation for circle bundles: If $\varphi:F\to F$ is a homeomorphism, the $F(\varphi):=(F\times [0,1])/\!\!\sim$, where $(x,0)$ is identified with $(\varphi(x),1)$. Assuming $F$ is oriented and $\varphi$ is orientation-preserving we endow $F\times [0,1]$ with a product orientation, which induces one on $F(\varphi)$. In this situation we have $-F(\varphi)\cong F(\varphi^{-1})$.

\begin{prop}[Splitter]\label{prop:splitter}
There is a compact, oriented, smooth, aspherical $4$-manifold $\rm{Splitter}$ with  and $\chi=0$ and aspherical and $\pi_{1}$-injective boundary consisting of 
\[\partial \ {\rm{Splitter} }=  T^{2}(\Phi^{-1})\sqcup T^{2}(\tau)\sqcup T^{2}(\tau)\sqcup T^{2}(\tau)\]
where $\Phi$ is the monodromy map in Proposition \ref{prop:core} and $\tau$ denotes a positive Dehn twist about a non-separating simple closed curve. 
\end{prop}

\begin{prop}[Connector]\label{prop:connector}
For any orientation-preserving diffeomorphism $\psi$ of the torus $T^{2}$ there is a compact, oriented, smooth, aspherical $4$-manifold $\rm{Conn}$ with $\chi=0$ and aspherical and $\pi_{1}$-injective boundary consisting of $$\partial\ {\rm {Conn}} =  T^{2}(\psi^{-1})\sqcup T^{2}(\psi^{-1})\sqcup T^{2}(\psi)\sqcup T^{2}(\psi)$$
\end{prop}

\begin{prop}[Cap]\label{prop:cap}
There is a compact, oriented, smooth, aspherical $4$-manifold $\rm{Cap}$ with $\chi=4$ and with aspherical and $\pi_{1}$-injective boundary consisting of $T^{2}(\tau)$,
where $\tau$ denotes a positive Dehn twist about a non-separating simple closed curve. \end{prop}

{\color{black}This is the most difficult step and uses a variation of a construction of Foozwell and Rubinstein \cite{FoozwellRubinstein2016}. It seems to require working with surfaces of genus greater than 2. We need to modify some details in order to squeeze out as small an Euler characteristic as possible. It would be quite interesting if one could reduce the Euler characteristic to $0$ in this construction.}

Reversing orientation one also has an object  with boundary  $T^{2}(\tau^{-1})$.

\paragraph{A warm-up construction}
To build a manifold with  $\chi=13$, take one copy of Core, one Splitter, and three copies of $\text{Cap}$ and glue them together appropriately, as indicated schematically in Figure \ref{fig:assembly}.
\begin{figure}[htbp]
\begin{tikzpicture}[scale=0.8]

\draw [ line width=1pt]
(2,5)
 .. controls ++(270:5) and ++(270:5) .. 
(8,5)
;
\draw [ line width=1pt]
(2,5)
 .. controls ++(270:1) and ++(270:1) .. 
(8,5)
;

\draw [dashed,  line width=1pt]
(2,5)
 .. controls ++(90:1) and ++(90:1) .. 
(8,5)
;

  \draw [ line width=1pt]
(2,5)
 .. controls ++(90:1) and ++(270:1) .. 
  (0,8.75) 
  ;

 \draw [ line width=1pt]
(8,5)
 .. controls ++(90:1) and ++(270:1) .. 
  (10,8.75) 
  ;
 \draw [ line width=1pt]
(0,8.75)
 .. controls ++(270:1) and ++(270:1) .. 
  (2,8.75) 
  ;
 \draw [dashed, line width=1pt]
(0,8.75)
 .. controls ++(90:1) and ++(90:1) .. 
  (2,8.75) 
  ;
 \draw [ line width=1pt]
(8,8.75)
 .. controls ++(270:1) and ++(270:1) .. 
  (10,8.75) 
  ;
 \draw [dashed, line width=1pt]
(8,8.75)
 .. controls ++(90:1) and ++(90:1) .. 
  (10,8.75) 
  ;
 \draw [ line width=1pt]
(4,8.75)
 .. controls ++(270:1) and ++(270:1) .. 
  (6,8.75) 
  ;
 \draw [dashed, line width=1pt]
(4,8.75)
 .. controls ++(90:1) and ++(90:1) .. 
  (6,8.75) 
  ;
   \draw [ line width=1pt]
(2,8.75)
 .. controls ++(270:1) and ++(270:1) .. 
  (4,8.75) 
  ;
 \draw [ line width=1pt]
(6,8.75)
 .. controls ++(270:1) and ++(270:1) .. 
  (8,8.75) 
  ;
 \draw [ line width=1pt]
(0,8.75)
 .. controls ++(90:4) and ++(90:4) .. 
  (2,8.75) 
  ;
 \draw [ line width=1pt]
(4,8.75)
 .. controls ++(90:4) and ++(90:4) .. 
  (6,8.75) 
  ;
   \draw [ line width=1pt]
(8,8.75)
 .. controls ++(90:4) and ++(90:4) .. 
  (10,8.75) 
  ;

\node at (5,3) {\large Core};  
\node at (5,2.4) {\large $\chi=1$};
\node at (5,7) {\large Splitter};  
\node at (5,6.4) {\large $\chi=0$};
\node at (1,10.5) {\large Cap};  
\node at (1,9.9) {\large $\chi=4$};
\node at (5,10.5) {\large Cap};  
\node at (5,9.9) {\large $\chi=4$};
\node at (9,10.5) {\large Cap};  
\node at (9,9.9) {\large $\chi=4$};

\end{tikzpicture}

\caption{Schematic of basic building blocks
}
\label{fig:assembly}
\end{figure}
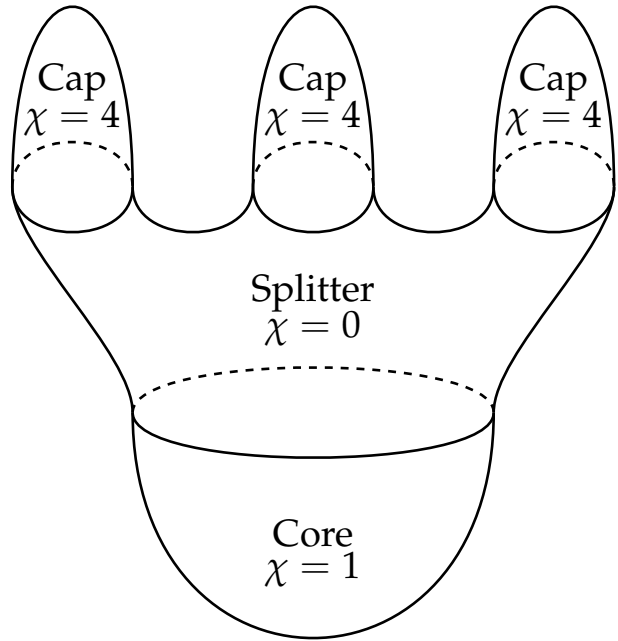

Assuming the preceding propositions, we can now prove the main theorem stated in the introduction.

\begin{thm}[Main Theorem]
For any integers  $k \ge 0$, there is a closed, oriented, aspherical, smooth $4$-manifold $X$ with 
\[
 \chi(X) = 
13 + 2k.
\]
\end{thm}

\begin{proof}
Start with $k+1$ copies of $\text{Core}$ and $k$ copies of $-\text{Core}$  and $2k+1$ $\text{Splitters}$. Attach one splitter to each Core or $-\text{Core}$ in such a way as to create an orientable $4$-manifold. Then glue them together using $2k$ connectors $\text{Conn}$ as needed, to create a compact, oriented, smooth, aspherical $4$-manifold with  $\chi=2k+1$, with boundary consisting of $3(k+1)$ copies of $T^{2}(\tau)$ and $3k$ copies of $T^{2}(\tau^{-1})$, where $\tau$ is a nontrivial positive Dehn twist. 

Identify the $3k$ copies of $T^{2}(\tau^{-1})$ with $3k$ of the copies of $T^{2}(\tau)$. 

Finally, cap off the remaining $3$ boundary components with $3$ copies of $\text{Cap}$. The result is a closed, oriented, aspherical, smooth $4$-manifold $X$ with 
\[
 \chi(X) = 
1+k +k +3\times 4 =13+2k\]
\end{proof}

\section{Core block}
\label{sec:core}
We now begin the process of constructing all the pieces.
\begin{thm}\label{thm:one}
There is a compact, oriented, aspherical, smooth $4$-manifold $V^{4}$ such that  $\chi(V^{4})=1$ and $\sigma(V^{4})= 1$ having aspherical and $\pi_{1}$-injective boundary $\partial V^{4}$ diffeomorphic to the torus bundle over the circle with monodromy 
\(
\left(\begin{array}{rr}0 & 1 \\-1&0 \end{array}\right).
\)
\end{thm}

\begin{proof}
Let $M$ denote the once-punctured $2$-torus with boundary $\partial M =S^{1}$, let $N$ be a small  collar neighborhood of $\partial M$, and set $M'=\overline{M\setminus N}$. Of course $M'\cong M$. Then $M$ and $M'$ are aspherical with $\pi_{1}$-injective, aspherical boundary. Consider $Z= \overline{ M\times M \setminus N\times N}$. Now $\partial Z$ consists of $S^{1}\times M'$ and $M'\times S^{1}$ separated by a tube which we identify with $\partial M\times\partial M \times I$. Each of these three pieces is aspherical and $\pi_{1}$-injective. (The full boundary of $M\times M$, or of $Z$, is not $\pi_{1}$-injective.)

We would like to identify the two pieces $\partial M\times M'$ and $M'\times \partial M$. The simplest way to do that would be via the interchange homeomorphism given by $(x,y)\mapsto (y,x)$. But this map is orientation-preserving and we need an orientation-reversing identification in order that the resulting identification space admits a well-defined orientation. To this end we use an identification of the form $\varphi:(x,y)\mapsto ({y},\overline{x})$, where $u\mapsto \overline{u}$ is an orientation-reversing involution of $M$ inducing a standard involution on $S^{1}$ with two fixed points. On the torus we may use $\overline{(z,w)}:=(w,z)$ and then delete the interior of an invariant disk neighborhood of a fixed point. 

We let $V=Z/\!\!\sim$ denote the resulting identification space. Now $\chi(Z)=1$ and $\chi(S^{1}\times M')=\chi(M'\times S^{1})=0$ by the product formula and therefore $\chi(V^{4})=1$ by the sum formula.

By Whitehead's theorem $V^{4}$ is aspherical. It follows from van Kampen's theorem that $\pi_{1}(\partial V)\to \pi_{1}(V)$ is injective.

It follows that $\partial V^{4}$ arises from identifying the two ends of $\partial M\times \partial M \times I$ via the map $\varphi$, which on homology is given by the matrix \(
\left(\begin{array}{rr}0 & 1 \\-1&0 \end{array}\right)
\), with respect to a suitable basis. We write $\partial V=T^{2}(\Phi)$, where $\Phi$ is the homeomorphism induced by this matrix.

\end{proof}

\section{Surface bundles}
The Splitters and Connectors both arise as torus bundles over a punctured disk. The Caps will be built by altering certain surface bundles, with fibers of genus 3.
\subsection{Basic properties}
Now a bundle over the circle $S^{1}$ with fiber $F$ is determined up to bundle isomorphism by its monodromy map $f:F\to F$.

Next we move to bundles with fiber $F$ over a punctured disk. Over each boundary curve one has a bundle over the circle with fiber $F$. The punctured disk is oriented as a submanifold of the Euclidean plane with a standard orientation and boundary curves are given the  orientation induced from that of the punctured disk.

Such a bundle is determined by its monodromy around the punctures, with the monodromy around the outer boundary being the inverse of the composition of the monodromies around the punctures, appropriately ordered.

\begin{lemma}
If $\varphi_{1}, \dots, \varphi_{n}$ are homeomorphisms of a closed oriented surface $F$ such that the composition $\varphi_{n}\dots \varphi_{1}=\text{ id}$, then there is a bundle over the $n$-times puncture $2$-sphere with fiber $F$ and monodromy around the punctures given by  $\varphi_{1}, \dots, \varphi_{n}$.
\end{lemma}

\begin{prop}\label{prop:genus3bundle}
If $\varphi$ is a positive Dehn twist about a non-separating simple closed curve in a closed, oriented surface $F_{h}$ of genus $h\ge 3$, then the $F_{h}$ bundle $F_{h}(\varphi)$ bounds an $F_{h}$ bundle over the once-punctured surface of genus $3$. The total space of the bundle has Euler characteristic $\chi=10(h-1)$.
\end{prop}
\noindent
We only need Proposition \ref{prop:genus3bundle} when $h=3$, in which case $\chi=10(h-1)=20$.
\begin{proof}
Since $h\ge 3$ there is a lantern configuration, as in Figure \ref{fig:lantern}, of $7$ non-separating simple closed curves in $F_{h}$ with corresponding positive Dehn twists $\varphi_{i}$ such that $$\varphi_{7}^{-1}\varphi_{6}^{-1}\varphi_{5}^{-1}\varphi_{4}\varphi_{3}\varphi_{2}\varphi_{1} \simeq\text{ id }.$$
For our purposes we understand the punctured disk to be embedded in the surface $F_{h}$ in such a way that all boundary curves are non-separating curves in $F_{h}$.

\begin{figure}[htbp]
\begin{center}

\begin{tikzpicture}[scale=1]
\draw [red,ultra thick] (0.5,2) circle [radius=0.45];;
\draw [red,ultra thick] (3,6) circle [radius=0.45];;
\draw [red,ultra thick] (5.5,2) circle [radius=0.45];;
\draw [red,ultra thick] (3,3.5) ellipse (4.5 and 4.5);;
\draw [blue,ultra thick] (3,2) ellipse (3.7 and 1.3);;
\draw[rotate around={58:(1.85,3.7)}, blue, ultra thick]  (2,4) ellipse (3.05 and 1.2);;
\draw[rotate around={-58:(3.3,2.0)}, blue, ultra thick]   (2,4) ellipse (3.2 and 1.2);;
\node at (0.5,2) {$C_{1}$};
\node at (5.5,2) {$C_{3}$};
\node at (3,6) {$C_{2}$};
\node at (-0.2,-0.2) {$C_{4}$};
\node at (1,4.5) {$C_{5}$};
\node at (3.2,1.1) {$C_{7}$};
\node at (5.1,4.5) {$C_{6}$};
\end{tikzpicture}
\caption{Lantern Curves $C_{1}, C_{2}, C_{3}, C_{4}, C_{5}, C_{6}, C_{7}$ yielding the lantern relation $\tau_{C_{4}}\tau_{C_{3}}\tau_{C_{2}}\tau_{C_{1}}=\tau_{C_{7}}\tau_{C_{6}}\tau_{C_{5}}$. These 7 curves are understood to lie on a 3-times punctured disk, which is embedded in a closed surface so that the curves are homologically non-trivial and homologically distinct.}
\label{fig:lantern}
\end{center}
\end{figure}
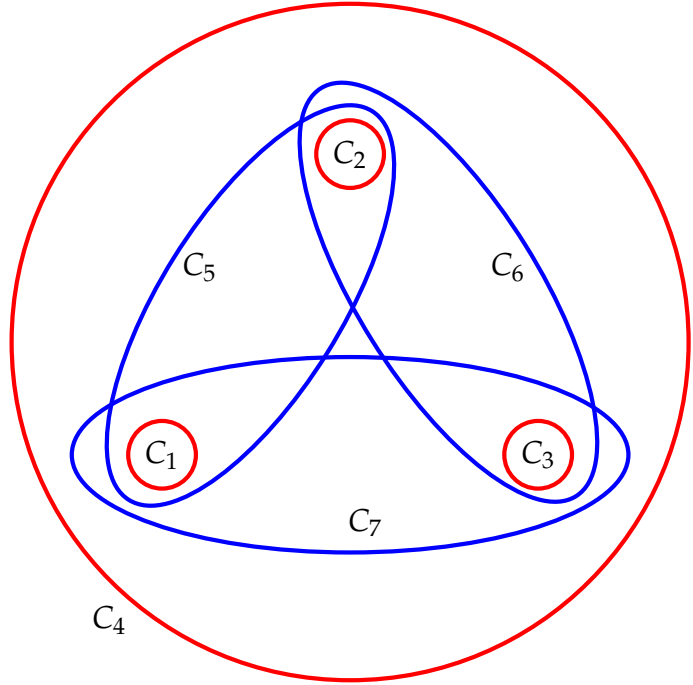

Corresponding to this identity there is an $F_{h}$ bundle over the seven-times punctured $2$-sphere with boundary monodromies given by $\varphi_{1}, \varphi_{2}, \varphi_{3}, \varphi_{4}, \varphi_{5}^{-1}, \varphi_{6}^{-1}, \varphi_{7}^{-1}$. 
Then we can identify six of the boundary components in pairs, 
$$F_{h}(\varphi_{2}), F_{h}(\varphi_{5}^{-1})  \text{ and }  F_{h}(\varphi_{3}), F_{h}(\varphi_{6}^{-1})  \text{ and }  F_{h}(\varphi_{4}), F_{h}(\varphi_{7}^{-1})$$
resulting in the desired manifold with boundary $F_{h}(\varphi_{1})$.

The Euler characteristic is immediate from the product formula for bundles:   $(2-7)(2-2h)=10(h-1)$. 
\end{proof}
The use of the lantern relation above is inspired by a somewhat similar use by Foozwell and Rubinstein. Its essential feature is that it involves an \emph{odd} number of Dehn twists and, indeed, has weight one, so that a composition of $n$ twists equals a composition of $n+1$ twists (along  non-separating curves).

\begin{lemma}\label{lemma:genus3cobordism}
If $\tau$ is a Dehn twist about a non-separating simple closed curve in a surface $F_{h}$ of genus $h\ge 2$, then there is a Dehn twist $\tau'$ about a non-separating simple closed curve on the torus $T^{2}$ and a Haken cobordism between $F_{h}(\tau)$ and $T^{2}(\tau')$ with $\chi=0$.
\end{lemma}
\begin{proof}
In the interior of a genus $h$ handlebody construct a simple closed curve that goes over each handle geometrically once and is suitably complicated so that the complement of the interior of a regular neighborhood gives a Haken cobordism $C^{3}$ between the genus $h$ surface and a torus. Do this in such a way that there is an annulus $A$ with one boundary a non-separating curve on $T^{2}$ and the other boundary a nonseparating curve on $F_{h}$. See the accompanying diagram.\footnote{That this curve has the desired properties follows from Oertel's analysis \cite{Oertel1984} of closed incompressible surfaces in Montesinos knot complements, for example. In particular, these surfaces bound handlebodies on the knot side.
}
Twisting around the annulus $A$ gives a diffeomorphism $\Psi$ of the cobordism that on each end is a Dehn twist about a nonseparating simple closed curve. The mapping torus of this diffeomorphism is a bundle $C^{3}(\Psi)$ over the circle. As such it is a Haken cobordism between $F_{1}(\Psi_{1})$ and $F_{h}(\Psi_{h})$, where $\Psi_{1}$ and $\Psi_{h}$ are the diffeomorphisms induced on each end of the cobordism. As a bundle over the circle it has  $\chi=0$ by the multiplicativity property of the Euler characteristic. \end{proof}

\begin{figure}[htbp]\label{fig:genus3cobordism}
\begin{center}
\begin{tikzpicture}[scale=0.7]
\draw [line width=1pt]
  (-5, 8) 
  .. controls ++(90:-2) and ++(180:1) .. 
  (-3, 4) 
  .. controls ++(180:-2) and ++(180:2) .. 
  (0, 5.5)
 .. controls ++(180:-2) and ++(180:2) .. 
  (3, 4) 
  .. controls ++(180:-2) and ++(180:2) .. 
   (6,5.5) 
   .. controls ++(180:-2) and ++(180:2) ..
    (9, 4) 
    .. controls ++(180:-1) and ++(270:2) .. 
    (11, 8)
    ;
    \draw  [line width=1pt]
  (-5, 8) 
  .. controls ++(90:2) and ++(180:1) .. 
  (-3, 12) 
  .. controls ++(180:-2) and ++(180:2) .. 
  (0, 10.5)
 .. controls ++(180:-2) and ++(180:2) .. 
  (3, 12)
  .. controls ++(180:-2) and ++(180:2) .. 
   (6,10.5) 
   .. controls ++(180:-2) and ++(180:2) .. 
    (9, 12) 
    .. controls ++(180:-1) and ++(90:2) .. 
    (11, 8)
    ;
        \draw  [line width=1pt]
  (-3, 9) 
  .. controls ++(60:-1) and ++(120:1) .. 
  (-3, 7) 
    ;
            \draw  [line width=3pt]
  (-3, 9) 
  .. controls ++(60:1) and ++(120:-1) .. 
  (-3, 7) 
    ;

        \draw  [line width=1pt]
  (3, 9) 
  .. controls ++(60:-1) and ++(120:1) .. 
  (3, 7) 
    ;
            \draw  [line width=3pt]
  (3, 9.) 
  .. controls ++(60:1) and ++(120:-1) .. 
  (3, 7) 
    ;
        \draw  [line width=1pt]
  (9, 9) 
  .. controls ++(60:-1) and ++(120:1) .. 
  (9, 7) 
    ;
       \draw  [line width=3pt]
  (9, 9) 
  .. controls ++(60:1) and ++(120:-1) .. 
  (9, 7) 
    ;
           \draw [color=red, line width=4pt]
  (-.05, 7.05) 
  .. controls ++(300:2) and ++(130:-2.5) .. 
  (6.25, 7.95) 
    ;
               \draw  [color=red, line width=4pt]
  (6, 8.4) 
  .. controls ++(120:1.5) and ++(90:3) .. 
  (10,8) 
   .. controls ++(270:2) and ++(135:-2) .. 
  (6.55,6.95) 
;
              \draw  [color=red, line width=4pt]
  (6.25, 7.25) 
  .. controls ++(135:1) and ++(130:-0.9) .. 
  (6.15,8.85) 
 ;
           \draw  [color=red, line width=4pt]
  (5.9,9.2) 
  .. controls ++(135:1) and ++(135:2) .. 
  (-0.2, 8.8) 
    ;
    
           \draw  [color=red, line width=4pt]
  (-0.5,9.5) 
  .. controls ++(135:1) and ++(75:2) .. 
  (-3.8, 8.8) 
  ;
            \draw  [color=red, line width=4pt]
     (-4,8)          
   .. controls ++(270:3) and ++(135:-1.5) .. 
(0.2,8.5)
    ;
              \draw  [color=red, line width=4pt]
  (-0.5, 7.5) 
  .. controls ++(135:1) and ++(130:-0.9) .. 
  (0,9.23) 
 ;
 \draw [blue, line width=1pt]
  (-5, 8) 
  .. controls ++(90:1) and ++(90:1) .. 
  (-3.4,8) 
  ;
   \draw  [blue, line width=1pt,dashed]
  (-3.4,8) 
    .. controls ++(90:-1) and ++(90:-1) .. 
  (-5, 8) 
  ;
\draw [blue, line width=1.5pt]
  (-4.2, 8) 
  .. controls ++(90:0.3) and ++(90:0.3) .. 
  (-3.8,8) 
  ;
   \draw  [blue, thick,dashed]
  (-3.8,8) 
    .. controls ++(90:-0.3) and ++(90:-0.3) .. 
  (-4.2, 8) 
  ;
\end{tikzpicture}

\caption{Haken cobordism from genus 3 to genus 1, showing a twist annulus.}
\end{center}
\end{figure}
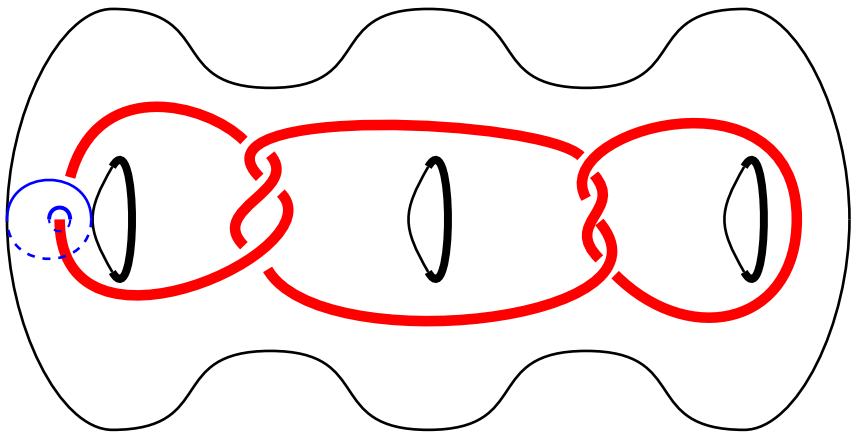

\begin{cor}
If $\varphi$ is a positive Dehn twist about a non-separating simple closed curve on the torus $T^{2}$, then the $T^{2}$ bundle $T^{2}(\varphi)$ bounds a Haken $4$-manifold of  Euler characteristic $\chi=20$.
\end{cor}
\begin{proof}
Attach the  cobordism from Lemma \ref{lemma:genus3cobordism} to the boundary of the bundle produced in Proposition \ref{prop:genus3bundle}, using a fiber surface of genus $h=3$.
\end{proof}
\section{Splitters and connectors}
The required splitters and connectors all arise as certain torus bundles over a punctured disk.
\label{sec:splitters}
\begin{lemma}
If $\varphi$ and $\psi$ are positive Dehn twists about non-separating simple closed curves on the torus, then $T^{2}(\varphi)$ and $T^{2}(\psi)$ are orientation-preserving diffeomorphic.
\end{lemma}
\begin{proof}
There is an orientation-preserving diffeomorphism taking any given non-separating  simple closed curve to any other one. Such a diffeomorphism induces a diffeomorphism of the corresponding bundles over the circle.
\end{proof}
\begin{lemma}
If $\psi$ is a positive Dehn twist about non-separating simple closed curve on the torus, then $-T^{2}(\psi)$ and $T^{2}(\psi^{-1})$ are orientation-preserving diffeomorphic.
\end{lemma}
\begin{proof}
For any nontrivial simple closed curve on the torus there is an orientation reversing diffeomorphism that is the identity on the simple closed curve. Such a diffeomorphism induces a diffeomorphism of the corresponding bundles over the circle.
\end{proof}

\begin{cor}[Connector]
For any orientation-preserving diffeomorphism $\psi$ of the torus $T^{2}$ there is a compact, oriented, smooth, aspherical $4$-manifold $\text{Conn}$ with $\chi=0$ and aspherical and $\pi_{1}$-injective boundary consisting of $$T^{2}(\psi^{-1})\sqcup T^{2}(\psi^{-1})\sqcup T^{2}(\psi)\sqcup T^{2}(\psi).$$
\end{cor}

\begin{lemma}
The transformation of the torus given by
\(
\Phi=\left(\begin{array}{rr}0 & 1 \\-1&0 \end{array}\right)
\)
is the composition of three (right-handed) Dehn twists.
\end{lemma}
\begin{proof}
\[
\left(\begin{array}{rr}0 & 1 \\-1&0 \end{array}\right) = 
\tau_{a}\tau_{b}\tau_{a}=
\left(\begin{array}{rr}1 & 1 \\0&1 \end{array}\right)
\left(\begin{array}{rr}1 & 0 \\-1&1 \end{array}\right)
\left(\begin{array}{rr}1 & 1 \\0&1 \end{array}\right).
\]
\end{proof}

\begin{cor}[Splitter]
If $\Phi=\left(\begin{array}{rr}0 & 1 \\-1&0 \end{array}\right)$, then there is an oriented aspherical 4-manifold with $\pi_{1}$-injective boundary
$$T^{2}(\tau_{a})\sqcup T^{2}(\tau_{b})\sqcup  T^{2}(\tau_{a}) \sqcup -T^{2}(\Phi)$$
 and $\chi=0$.
\end{cor}
\begin{proof}
The torus bundle over the three-times punctured disk with puncture monodromies $\tau_{a}, \tau_{b}, \tau_{a}$ has $\chi=0$ since it is a bundle with fiber of Euler characteristic $0$. \end{proof}
\section{Caps}\label{sec:caps}
We have the general question about when a closed aspherical $3$-manifold bounds a compact Haken $4$-manifold (with boundary being $\pi_{1}$-injective), with  Euler characteristic as small as possible. We are particularly interested in this question for the torus bundle over the circle with monodromy a nontrivial Dehn twist, and disjoint unions of such manifolds. We adopt the short-hand notation $W^{4}[\alpha,\beta,\gamma]$ to denote the bundle with fiber $F$ over the twice-punctured disk with monodromies $\alpha,\beta,\gamma:F\to F$ with $\alpha\beta\gamma=\text{id}$.

\begin{prop}[Foozwell-Rubinstein torus trick \cite{FoozwellRubinstein2016}, Lemma 2.3 interpreted in our context]
Let $\varphi,\tau_{c}:F_{h}\to F_{h}$ be orientation-preserving diffeomorphisms, with $\tau_{c}$ a Dehn twist about a simple closed curve $c$ in $F_{h}$ pointwise fixed by $\varphi$. Then there is a $4$-manifold $\widehat{W}^{4}$ obtained from $F_{h}(\varphi)\times I$ by attaching a torus handle $T^{2}\times I\times I$ along a pair of parallel copies  of framed tori $T^{2}\times I\times \{0,1\}$ in $F_{h}(\varphi)\times \{1\}$ such that the boundary of $\widehat{W}^{4}$ consists of  $F_{h}(\varphi)$, $T^{2}(\psi)$, and $F_{h}(\varphi^{-1}\tau_{c}^{-1})$.  Here $\psi$ is a Dehn twist about a non-separating simple closed curve on the torus.
\end{prop}
We view  $\widehat{W}^{4}$ as a modification of the  surface bundle $W^{4}=W^{4}[\varphi,\tau_{c}, \varphi^{-1}\tau_{c}^{-1}] $ over a twice-punctured disk. 
Note that while the Euler characteristic $\chi(W^{4})=2h-2$, we always have $\chi(\widehat{W}^{4})=0$, since $\widehat{W}^{4}$ is constructed by gluing together pieces of Euler characteristic $0$ with intersections of Euler characteristic $0$. In particular the application of the torus trick here reduces the Euler characteristic by $2h-2$.

\begin{prop}[Caps]
If $\tau$ is a Dehn twist along a non-separating simple closed curve on the $2$-torus $T^{2}$, then the torus bundle $T^{2}(\tau)$ bounds a compact oriented Haken $4$-manifold Cap with $\pi_{1}$-injective boundary and Euler characteristic $\chi=4$.
\end{prop}
\begin{proof}
We start with the analogous result from Proposition \ref{prop:genus3bundle} for a surface of genus 3 (where the resulting $4$-manifold could be chosen to be a surface bundle) and show how to modify that construction to achieve $\chi=4$. That construction started with an $F_{h}$-bundle over the $7$-times punctured $2$-sphere with puncture monodromies $\varphi_{1}, \varphi_{2}, \dots, \varphi_{7}$ such that $$\varphi_{7}^{-1}\varphi_{6}^{-1}\varphi_{5}^{-1}\varphi_{4}\varphi_{3}\varphi_{2}\varphi_{1}=1.$$ Then one identified six of the resulting boundary components in pairs. Recall that the first four are positive Dehn twists and the last three are negative Dehn twists, all about non-separating simple closed curves.  Instead of immediately identifying the boundary bundles in pairs as we did before, we alter four boundary pieces into torus bundles in such a way that each such alteration reduces the Euler characteristic by $4$, using the Foozwell-Rubinstein torus trick. 

We break the $7$-times punctured $2$-sphere into $5$ pairs of pants, with corresponding monodromies, and corresponding preimages
\begin{description}
\item[ $W_{2}^{4}$] $W^{4}[\varphi_{1}, \varphi_{2}, \varphi_{1}^{-1}\varphi_{2}^{-1}],$\\
\item[ $W_{3}^{4}$] $W^{4}[\varphi_{2}\varphi_{1},\varphi_{3}, \varphi_{1}^{-1}\varphi_{2}^{-1}\varphi_{3}^{-1}]$\\
\item[ $W_{4}^{4}$] $W^{4} [\varphi_{3}\varphi_{2}\varphi_{1},\varphi_{4}, \varphi_{1}^{-1}\varphi_{2}^{-1}\varphi_{3}^{-1}\varphi_{4}^{-1}]$\\
\item[ $W_{5}^{4}$] $W^{4} [\varphi_{4}\varphi_{3}\varphi_{2}\varphi_{1},\varphi_{5}^{-1}, \varphi_{1}^{-1}\varphi_{2}^{-1}\varphi_{3}^{-1}\varphi_{4}^{-1}\varphi_{5}],$\\
\item[ $W_{6}^{4}$] $W^{4}[\varphi_{5}^{-1 }\varphi_{4}\varphi_{3}\varphi_{2}\varphi_{1},\varphi_{6}^{-1}, \varphi_{1}^{-1}\varphi_{2}^{-1}\varphi_{3}^{-1}\varphi_{4}^{-1}\varphi_{5}\varphi_{6}]$
\end{description}

We can apply the Torus Trick to the first four of these five manifolds: $\varphi_{1}$ fixes the support of $\varphi_{2}$; $\varphi_{2}\varphi_{1}$ fixes the support of $\varphi_{3}$;  $\varphi_{3}\varphi_{2}\varphi_{1}$ fixes the support of $\varphi_{4}$; and $\varphi_{4}\varphi_{3}\varphi_{2}\varphi_{1}$ fixes the support of $\varphi_{5}$.

Then  glue back together the altered pieces, in place of the originals, to form a connected manifold. Finally attach torus bundle cobordisms to the remaining non-identified boundaries and identify all remaining pairs, leaving the desired manifold with $\chi=4$. The construction is summarized in the schematic drawing in Figure \ref{fig:pants}.
\end{proof}
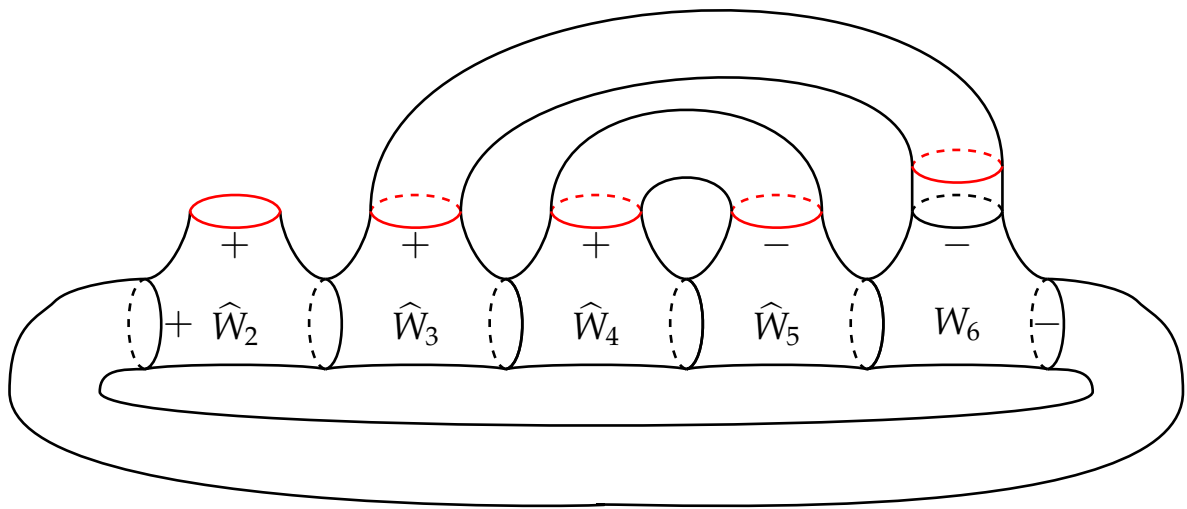
\begin{figure}
\begin{center}
\begin{tikzpicture}[scale=0.3]
\draw [dashed, line width=1pt]
(2,1)
 .. controls ++(165:1) and ++(195:1) .. 
  (2, 5) 
  ;
  \draw [line width=1pt]
    (2, 5) 
   .. controls ++(-15:1) and ++(15:1) .. 
   (2,1)
  .. controls ++(15:1) and ++(-15:-1) .. 
  (10, 1) 
  ;
  \draw [dashed, line width=1pt]
  (10,1)
    .. controls ++(165:1) and ++(195:1) .. 
    (10,5)
  ;
  \draw [line width=1pt]
    (10,5)
        .. controls ++(-15:1) and ++(15:1) .. 
    (10,1)
  ;
\draw [line width=1pt]
(2,5)
 .. controls ++(0:1) and ++(270:1) .. 
  (4, 8) 
;
\draw [line width=1pt]
(10,5)
 .. controls ++(180:1) and ++(270:1) .. 
  (8, 8) 
;
\draw [color=red, line width=1pt]
(4,8)
 .. controls ++(-75:1) and ++(255:1) .. 
  (8, 8) 
  .. controls ++(105:1) and ++(75:1) .. 
  (4,8) 
;
\node at (6,3) {\large$\widehat{W}_{2}$};
\node at (3.5,3) {\large$+$};
\node at (6,6.5) {\large$+$};

\draw [line width=1pt]
   (12-2,1)
  .. controls ++(15:1) and ++(-15:-1) .. 
  (20-2, 1) 
  ;
  \draw [dashed, line width=1pt]
  (20-2,1)
    .. controls ++(165:1) and ++(195:1) .. 
    (20-2,5)
  ;
  \draw [line width=1pt]
    (20-2,5)
        .. controls ++(-15:1) and ++(15:1) .. 
    (20-2,1)
  ;
\draw [line width=1pt]
(12-2,5)
 .. controls ++(0:1) and ++(270:1) .. 
  (14-2, 8) 
;
\draw [line width=1pt]
(20-2,5)
 .. controls ++(180:1) and ++(270:1) .. 
  (18-2, 8) 
;
\draw [color=red, line width=1pt]
(14-2,8)
 .. controls ++(-75:1) and ++(255:1) .. 
  (18-2, 8) 
  ;
  \draw [dashed, color=red, line width=1pt]
    (18-2, 8) 
  .. controls ++(105:1) and ++(75:1) .. 
  (14-2,8) 
;
\node at (16-2,3) {\large$\widehat{W}_{3}$};
\node at (14,6.5) {\large$+$};
\draw [line width=1pt]
  (22-4, 5) 
   .. controls ++(-15:1) and ++(15:1) .. 
   (22-4,1)
  .. controls ++(15:1) and ++(-15:-1) .. 
  (30-4, 1) 
  ;
  \draw [dashed, line width=1pt]
  (30-4,1)
    .. controls ++(165:1) and ++(195:1) .. 
    (30-4,5)
  ;
  \draw [line width=1pt]
    (30-4,5)
        .. controls ++(-15:1) and ++(15:1) .. 
    (30-4,1)
  ;
\draw [line width=1pt]
(22-4,5)
 .. controls ++(0:1) and ++(270:1) .. 
  (24-4, 8) 
;
\draw [line width=1pt]
(30-4,5)
 .. controls ++(180:1) and ++(270:1) .. 
  (28-4, 8) 
;
\draw [color=red, line width=1pt]
(24-4,8)
 .. controls ++(-75:1) and ++(255:1) .. 
  (28-4, 8) 
  ;
  \draw [dashed, color=red, line width=1pt]
    (28-4, 8) 
  .. controls ++(105:1) and ++(75:1) .. 
  (24-4,8) 
;
\node at (26-4,3) {\large$\widehat{W}_{4}$};
\node at (22,6.5) {\large$+$};

\draw [line width=1pt]
  (32-6, 5) 
   .. controls ++(-15:1) and ++(15:1) .. 
   (32-6,1)
  .. controls ++(15:1) and ++(-15:-1) .. 
  (40-6, 1) 
  ;
  \draw [dashed, line width=1pt]
  (40-6,1)
    .. controls ++(165:1) and ++(195:1) .. 
    (40-6,5)
  ;
  \draw [line width=1pt]
    (40-6,5)
        .. controls ++(-15:1) and ++(15:1) .. 
    (40-6,1)
  ;
\draw [line width=1pt]
(32-6,5)
 .. controls ++(0:1) and ++(270:1) .. 
  (34-6, 8) 
;
\draw [line width=1pt]
(40-6,5)
 .. controls ++(180:1) and ++(270:1) .. 
  (38-6,8) 
;
\draw [color=red, line width=1pt]
(34-6,8)
 .. controls ++(-75:1) and ++(255:1) .. 
  (38-6, 8) 
    ;
  \draw [dashed, color=red, line width=1pt]
    (38-6, 8) 
  .. controls ++(105:1) and ++(75:1) .. 
  (34-6,8) 
;
\node at (36-6,3) {\large$\widehat{W}_{5}$};
\node at (30,6.5) {\large$-$};
\draw [line width=1pt]
  (42-8, 5) 
   .. controls ++(-15:1) and ++(15:1) .. 
   (42-8,1)
  .. controls ++(15:1) and ++(-15:-1) .. 
  (50-8, 1) 
  ;
  \draw [dashed, line width=1pt]
  (50-8,1)
    .. controls ++(165:1) and ++(195:1) .. 
    (50-8,5)
  ;
  \draw [line width=1pt]
    (50-8,5)
        .. controls ++(-15:1) and ++(15:1) .. 
    (50-8,1)
  ;
\draw [line width=1pt]
(42-8,5)
 .. controls ++(0:1) and ++(270:1) .. 
  (44-8, 8) 
;
\draw [line width=1pt]
(50-8,5)
 .. controls ++(180:1) and ++(270:1) .. 
  (48-8,8) 
;
\draw [line width=1pt]
(44-8,8)
 .. controls ++(-75:1) and ++(255:1) .. 
  (48-8, 8) 
  ;
  \draw [dashed, line width=1pt]
    (48-8, 8) 
  .. controls ++(105:1) and ++(75:1) .. 
  (44-8,8) 
;
\node at (46-8,3) {\large${W}_{6}$};
\node at (42,3) {\large$-$};
\node at (38,6.5) {\large$-$};

\draw [line width=1pt]
(44-8,8)
 .. controls ++(90:1) and ++(-90:1) .. 
  (44-8, 10) 
  ;
  \draw [line width=1pt]
  (48-8,8)
  .. controls ++(90:1) and ++(-90:1) .. 
  (48-8,10) 
;
\draw [color=red, line width=1pt]
(44-8,10)
 .. controls ++(-75:1) and ++(255:1) .. 
  (48-8, 10) 
  ;
  \draw [dashed, color=red, line width=1pt]
    (48-8, 10) 
  .. controls ++(105:1) and ++(75:1) .. 
  (44-8,10) 
;
\draw [line width=1pt]
(2,1)
 .. controls ++(180:1) and ++(90:1) .. 
 (0,0)
  .. controls ++(270:2) and ++(270:2) .. 
 (44,0)
 .. controls ++(90:1) and ++(0:1) .. 
 (42,1)
     ;
\draw [line width=1pt] 
(2,5)
 .. controls ++(180:1) and ++(45:1) .. 
 (-2,4)
  .. controls ++(-135:1) and ++(90:3) .. 
 (-4,0)
  .. controls ++(270:6) and ++(-0:4) .. 
  (22,-5)
  .. controls ++(0:4) and ++(-90:6) ..
 (48,0)
  .. controls ++(90:3) and ++(135:-1) .. 
 (46,4)
 .. controls ++(135:1) and ++(0:1) .. 
 (42,5)
     ;

\draw [line width=1pt] 
(20,8)
 .. controls ++(90:6) and ++(90:6) .. 
 (32,8)
;
\draw [line width=1pt] 
(24,8)
 .. controls ++(90:2) and ++(90:2) .. 
 (28,8)
;


\draw [line width=1pt] 
(12,8)
 .. controls ++(90:11) and ++(90:10) .. 
 (40,10)
;
\draw [line width=1pt] 
(16,8)
 .. controls ++(90:7) and ++(90:6) .. 
 (36,10)
;

\end{tikzpicture}
\end{center}
\caption{Schematic of a Cap formed from five pieces glued together with the final three gluings indicated by tubes. The $\pm$ signs indicate orientations, which need to cancel as boundary components are glued together.}
\label{fig:pants}
\end{figure}
As for working on the fifth piece $W_{6}^{4}$ to try to reduce the Euler characteristic one more time,
we note that it is the same as
\[
W^{4}[\varphi_{6}\varphi_{7}, \varphi_{6}^{-1}, \varphi_{7}^{-1}].
\]
The problem is that the two (negative) Dehn twists are about intersecting simple closed curves, with geometric intersection number 2 and algebraic intersection number 0, rather than disjoint simple closed curves.


\section{Comments on signatures}
We originally intended to treat the signatures of the manifolds constructed here as well, but found that the needed details and attention to orientations overwhelmed the main point, and in the end nothing unexpected was discovered.   The challenge is to compute the signatures of the various pieces, in certain cases applying the Meyer signature formula (a specialization of Wall's non-additivity invariant) and Novikov additivity. According to our calculations the Cores have signature 1 as do the Splitters, while the Caps have signature $-1$, all with suitable choice of orientations. If we did it correctly, the manifolds constructed here with $\chi=13+2n$ all have signature $\sigma=1$ (with appropriate orientation). More generally one can construct manifolds with $\sigma=s$ and $\chi=13s+2n$ for any integers $s,n\ge 0$. For the latter one needs one additional building block that caps off groups of six copies  of $T^{2}(\tau_{a})$ together and has signature $-4$ and Euler characteristic 0. Further details are omitted. 

The most likely conjecture about the broader geography of aspherical 4-manifolds appears to be the following:
\begin{conj}
If $X^{4}$ is a closed, oriented, aspherical $4$-manifold, then $\chi(X^{4})\ge \left|\sigma(X^{4})\right|$.
\end{conj}
A yet stronger conjecture, known to hold for surface bundles over surfaces by a result of Hamenst\"adt \cite{Hamenstaedt2012}, is as follows:
\begin{conj}
If $X^{4}$ is a closed, oriented, aspherical $4$-manifold, then $\chi(X^{4})\ge 3\left|\sigma(X^{4})\right|$.
\end{conj}
The results of this papers are consistent with either of these conjectures. A closed, oriented, aspherical $4$-manifold with $\chi=1$, however, would violate at least the second one.
\bibliographystyle{amsalpha}
\providecommand{\bysame}{\leavevmode\hbox to3em{\hrulefill}\thinspace}
\providecommand{\MR}{\relax\ifhmode\unskip\space\fi MR }
\providecommand{\MRhref}[2]{%
  \href{http://www.ams.org/mathscinet-getitem?mr=#1}{#2}
}
\providecommand{\href}[2]{#2}

\end{document}